\documentclass[review,hidelinks,onefignum,onetabnum]{siamart220329}

\usepackage{lipsum}
\usepackage{amsfonts}
\usepackage{graphicx}
\usepackage{epstopdf}
\usepackage{algorithmic}
\usepackage{listings}
\ifpdf
  \DeclareGraphicsExtensions{.eps,.pdf,.png,.jpg}
\else
  \DeclareGraphicsExtensions{.eps}
\fi

\newsiamremark{remark}{Remark}
\newsiamremark{hypothesis}{Hypothesis}
\crefname{hypothesis}{Hypothesis}{Hypotheses}
\newsiamthm{claim}{Claim}

\headers{Spectrum of MATLAB's magic squares}{Hariprasad Manjunath and Sivaram Ambikasaran}

\title{Spectrum of MATLAB's magic squares\thanks{Submitted to the editors May $25^{th}$, $2022$.
\funding{This work was funded by none.}}}

\author{Hariprasad Manjunath\thanks{Indian Institute of Information Technology Dharwad
  (\email{hariprasad@iiitdwd.ac.in}, \url{https://www.iiitdwd.ac.in/hari.php}).}
\and Sivaram Ambikasaran\thanks{Indian Institute of Technology Madras 
  (\email{sivaambi@smail.iitm.ac.in}, \url{http://sivaramambikasaran.com/}).}}

\usepackage{amsopn}

\usepackage{amsmath,amssymb}

\usepackage{placeins}

\usepackage[most]{tcolorbox}

\newcommand{\Rb}{\mathbb{R}}
\newcommand{\Cb}{\mathbb{C}}
\newcommand{\Zb}{\mathbb{Z}}
\newcommand{\Ib}{\mathbb{I}}
\newcommand{\Ob}{\mathbb{O}}
\newcommand{\bkt}[1]{\left(#1\right)}
\newcommand{\abs}[1]{\left\lvert#1\right\rvert}
\newcommand{\magn}[1]{\left\lVert#1\right\rVert}
\newcommand{\dsum}{\displaystyle\sum}

\newcommand{\red}[1]{{\color{red}#1}}
\newcommand{\blue}[1]{{\color{blue}#1}}

\newcommand{\brown}[1]{{\color{brown}#1}}
\newcommand{\cyan}[1]{{\color{cyan}#1}}
\newcommand{\teal}[1]{{\color{teal}#1}}

\ifpdf
\hypersetup{
  pdftitle={Spectrum of MATLAB's magic squares},
  pdfauthor={Hariprasad Manjunath and Sivaram Ambikasaran}
}
\fi

\begin{document}
\nolinenumbers
\maketitle
\begin{abstract}
	This article looks at the eigenvalues of magicsquares generated by the MATLAB's magic($n$) function. The magic() function constructs doubly even ($n = 4k$) magic squares, singly even ($n = 4k+2$) magic squares and odd ($n = 2k+1$) magic squares using different algorithms. The doubly even magic squares are constructed by a criss-cross method that involves reflecting the entries of a simple square about the center. The odd magic squares are constructed using the Siamese method. The singly even magic squares are constructed using a lower-order odd magic square (Strachey method). We obtain approximations of eigenvalues of odd and singly even magic squares with error bounds. Further, the eigenvalues of doubly even magic squares are exactly obtained. The approximation of the spectra involves some interesting connections with the spectrum of g-circulant matrices and the use of Bauer-Fike theorem.
\end{abstract}
\section{Introduction}
Magic squares are given historical~\cite{anderson2001mathematical}, cultural~\cite{cammann1969islamic} and recreational~\cite{benefiel2012magic} importance. Matrix properties of the magic square are also studied~\cite{loly2009magic}\cite{van1990magic} for various applications in image processing~\cite{liu2012study}, privacy \cite{ranjani2017data},  statistics \cite{hunter2010some}, and geophysics \cite{adetokunbo20163d}. Spectral properties of some special magic squares have also been studied in the past~\cite{loly2009magic},~\cite{mattingly2000even} and ~\cite{lee2006linear}.

In this article, we discuss the eigenvalues of a class of magic squares. More specifically, we obtain the eigenvalues of magic squares generated by MATLAB's magic() command~\cite{moler2004numerical}. The magic() command uses the
\begin{enumerate}
	\item
	Siamese method~\cite{weisstein2002magic} for generating odd magic squares,
	\item
	Strachey method~\cite{ball2016mathematical} for generating singly even magic squares,
	\item
	A form of \emph{criss-cross} method for generating doubly even magic squares.
\end{enumerate}
The next section (Section~\ref{Notations}) provides the notations and definitions we will be working on within this article. In the subsequent sections (Sections~\ref{Odd},~\ref{singly_even},~\ref{doubly_even}), we prove theorems on the spectrum of the eigenvalues of odd, singly even and doubly even magic squares.

\section{Preliminaries}
\subsection{Notations and definitions}
\label{Notations}
The following notations and definitions are used throughout this article.
\begin{enumerate}
	\item
	$e_k$ is a unit column vector (of appropriate dimensions) whose $k^{th}$ entry is $1$.
	\item
	$\Ib_{p \times q}$ is a $p \times q$ matrix with all entries as $1$.
	\item
	$\Ib_{p}$ is a $p \times 1$ column vector with all entries as $1$. Note that $\Ib_{p \times 1} = \Ib_p$.
	\item
	$\Ob_{r \times s}$ is a $r \times s$ matrix with all entries as $0$.
	\item
	$\Ob_{r}$ is a $r \times 1$ column vector with all entries as $0$. Note that $\Ob_{r \times 1} = \Ob_r$.
	\item
	$I_t$ is a $t \times t$ identity matrix and $J_t$ is a $t \times t$ reverse identity matrix, i.e.,
	\begin{equation}
		I_t(i,j) = \begin{cases}
		1 & \text{ if }i=j,\\
		0 & \text{ otherwise}.
		\end{cases}\\
		J_t(i,j) = \begin{cases}
		1 & \text{ if }i+j=t+1,\\
		0 & \text{ otherwise}.
		\end{cases}
	\end{equation}
	\item
	A set $P$ is said to be $\epsilon$-approximated by set $Q$ if for each $p \in P$, there exists a $q(p) \in Q$ such that
	\begin{equation}
		\dfrac{\abs{p-q}}{\abs{q}} \leq \epsilon.
	\end{equation}
	\item
	A set $P$ is said to be $\epsilon$-well-approximated by set $Q$ if for each $p \in P$, there exists a $q(p) \in Q$ such that
	\begin{equation}
		\dfrac{\abs{p-q}}{\abs{q}} \leq \epsilon,
	\end{equation}
	and the map $q: P \mapsto Q$ is a bijection.
    \item
    $\omega$ will denote the $n^{th}$ root of unity given by $\omega = e^{-2 \pi i/n}$.
    \item
    The vector $v_j = \begin{bmatrix}
    1 & \omega^j & \omega^{2j} & \cdots & \omega^{(n-1)j}
    \end{bmatrix}^T \in \Rb^n$ will be termed as the $j^{th}$ Fourier mode. Note that $v_{n-j} = v_{-j}$ and $v_i^*v_j = n\delta_{ij}$.
    
    \item
    The $n \times n$ DFT matrix will be denoted by $F_n$ and is given by
    \begin{equation}
        F_n = \begin{bmatrix}
        v_0 & v_1 & \cdots & v_{n-1}
        \end{bmatrix}.
    \end{equation}
    	We have $ F_n\bkt{p,q}  = \omega^{\bkt{p-1}\bkt{q-1}} $ and $	F_n^*F_n = nI_{n} = F_nF_n^*$
 \item $C$ be a cycle permutation matrix given by 
 $C = \begin{bmatrix}
 \Ob_{1 \times (n-1)} & 1 \\
 I_{n \times n} & \Ob_{(n-1) \times 1}
 \end{bmatrix}.  $
 Note that pre-multiplying by $C$ permutes the rows along the forward cycle, while post-multiplying by $C$ permutes the columns along the reverse cycle.
 \item
	$M_n$ is the $n \times n$ magic square obtained from MATLAB's magic(n) command.
	\item
	$\mu_0$ is the magic sum $\dfrac{n\bkt{n^2+1}}2$.
	\item
	$\left\{\mu_i \right\}_{i=0}^{n-1}$ will denote the $n$ eigenvalues of $M_n$.
	\item
	We term $\bkt{\mu_0,\Ib_n}$ to be the trivial eigenpair.
 \end{enumerate}
 
 \subsection{Properties of magicsquares}
 In this section, we list some well-known properties of the magic squares.
\begin{itemize}
\item For any $n \times n$ magic square, $M_n$, we have $M_n \Ib_{n \times 1} = \dfrac{n\bkt{n^2+1}}2 \Ib_{n \times 1}$ and $\Ib_{1 \times n} M_n = \dfrac{n\bkt{n^2+1}}2 \Ib_{1 \times n}$.
\item The magic sum is the largest eigenvalue with multiplicity $1$.
	This Follows immediately from the Perron-Frobenius theorem.
\item 
	Every eigenvector of $M_n$ other than $\Ib_{n \times 1}$ is orthogonal to $\Ib_{n \times 1}$.
\begin{proof}
	This follows by noting that all one vector is also a left eigenvector, and left and right eigenvectors of a matrix are biorthogonal.
\end{proof}
\end{itemize}

 \subsection{Spectrum of circulant matrices} In this section we recall some spectral properties of circulant matrices. Let $R$ be a general circulant matrix.
The vectors $v_j$ are the eigenvectors of the circulant matrix $R$. The eigenvalue corresponding to the vector $v_j$ is given by
	\begin{equation}
		\lambda_j = \dfrac{v_j^*Rv_j}{v_j^*v_j} = \dfrac{v_j^*Rv_j}n.
	\end{equation}
If $\mathbf{a}$ is the first row of the circulant matrix $R$, we then have
	\begin{equation}
		\lambda_j = \mathbf{a}v_j.
	\end{equation}

This can be proved by looking at the $i^{\text{th}}$ entry of the vector $Rv_j$. We have  $$(R v_j)(i) = \mathbf{a} C^{n-(i-1)} v_j = \mathbf{a} \omega^{((i-1)j)} v_j = \mathbf{a}v_j \omega^{((i-1)j)} = (\mathbf{a}v_j) v_j(i) .$$

If $\Lambda_n \in \Cb^{n \times n}$ is the matrix of eigenvalues of $R$, i.e.,
	\begin{equation}
		\Lambda_n\bkt{i,j} = \begin{cases}
	\lambda_j & \text{ if }i=j,\\
	0 & \text{ otherwise},
	\end{cases}
	\end{equation}
	we then have
	\begin{equation}
		\Lambda_n = \dfrac{F_n^*RF_n}n,
    \end{equation}
    \begin{equation}
		R = \dfrac{F_n\Lambda_n F_n^*}n.
	\end{equation}

\subsection{g-circulant matrices}\label{gcirc}
Let $R_g$ is a $g$-row circulant matrix, where $g\in\{1,2,\ldots,n-1\}$, if for all $i,j,k \in \{1,2,\ldots,n\}$, we have
	\begin{equation}
		R_g\bkt{i,j} = R_g\bkt{\bkt{i+k}\bmod{n},\bkt{j+gk}\bmod{n}}.
	\end{equation}
	It is possible to have a decomposition $R_g = P_g R$, where $P_g$ is a $g$-row circulant matrix with first row being $e_1^T$,  $R$ is a circulant matrix having the same first row as $R_g$.
 Also g-row circulant matrix is a $s = (g^{-1} \mod n)$ column circulant matrix. Then we have $R_g = R P_s$ with $R$ being the circulant matrix with same column as $R_g$. And $P_s$ is an $s$ row circulant matrix with first column beign $e_1$.
 \subsection{Spectrum of reverse circulant matrix}\label{rc_eig} In this subsection, we obtain the eigenvalues and eigenvectors of a reverse circulant matrix.
 Let $R$ be the circulant matrix with eigenpairs given by $\bkt{\lambda_j,v_j}_{j=0}^{n-1}$, i.e.,
	\begin{equation}
		Rv_j = \lambda_j v_j.
	\end{equation}
	We now find the eigenpairs of the matrix $JR$, where $J$ is the flipped identity matrix. Note that $(\lambda_0, \mathbb{I}_{n \times 1})$ is also an eigenpair of the reverse circulant matrix. \\
    We have
	\begin{equation}
		JRv_j = \lambda_j Jv_j.
	\end{equation}
	We also have
	\begin{align}
		Jv_j & = \begin{bmatrix}\omega^{(n-1)j} & \omega^{(n-2)j} & \cdots & \omega^{2j} & \omega^{j} & 1\end{bmatrix}^T\\
		& = \omega^{-j} \begin{bmatrix}1 & \omega^{-j} & \cdots & \omega^{-(n-2)j} & \omega^{-(n-1)j} \end{bmatrix}^T\\
		& = \omega^{-j} v_{n-j}.
	\end{align}
	Note that $v_{n-j} = \overline{v_j}$ and $\lambda_{n-j} = \overline{\lambda_j}$. Hence, we obtain
	\begin{equation}
		JRv_j = \lambda_j Jv_j = \omega^{-j} \lambda_j v_{n-j}.
	\end{equation}
	Replacing $j$ by $n-j$, we obtain
	\begin{equation}
		JRv_{n-j} = \lambda_{n-j} Jv_{n-j} = \omega^{-\bkt{n-j}} \lambda_{n-j} v_{j} = \omega^{j} \lambda_{n-j} v_{j} = \overline{\omega^{-j} \lambda_{j}} v_{j}.
	\end{equation}
	Hence, we see that
	\begin{align}
		JR \bkt{v_j \pm \dfrac{\omega^{-j} \lambda_j}{\abs{\omega^{-j} \lambda_j}} v_{n-j}} & = \omega^{-j}\lambda_j v_{n-j} \pm \abs{\omega^{-j} \lambda_j} v_j,\\
		& = \mp \abs{\omega^{-j} \lambda_j} \bkt{v_j \pm \dfrac{\omega^{-j} \lambda_j}{\abs{\omega^{-j} \lambda_j}} v_{n-j}}.
	\end{align}
    We consider $j \leq \frac{n-1}{2}$ to get $n-1$ the eigenpairs (We are interested with odd $n$ in this article). 
	Hence, we obtain the $ n-1 $ eigenpairs of $JR$ to be
	$$\bkt{ \pm \abs{\lambda_j},v_j \mp \dfrac{\omega^{-j} \lambda_j}{\abs{\lambda_j}} v_{n-j}}   \text{  for } j \leq \frac{n-1}{2},$$
	where $\bkt{\lambda_j,v_j}$ are the eigenpairs of $R$. \\ Note that if we let $u_0 = \frac{1}{\sqrt{n}}\mathbb{I}$, and $u_j = \dfrac{v_j \pm \dfrac{\omega^{-j} \lambda_j}{\abs{\lambda_j}} v_{n-j}}{\magn{v_j \pm \dfrac{\omega^{-j} \lambda_j}{\abs{\lambda_j}} v_{n-j}}_2}$, \\
 then we have
	\begin{equation}
		u_j^*u_k = \delta_{j,k}.
	\end{equation}
    Hence, the eigenvectors of the reverse circulant matrix are orthonormal.

\section{Odd magic squares}
\label{Odd}
The first two odd magic squares generated by MATLAB's magic() command are given in Equation~\eqref{odd_magic}.
\begin{align}
M_3 = \begin{bmatrix}
\teal{8} & \red{1} & \blue{6} \\
\red{3} & \blue{5} & \teal{7} \\
\blue{4}  & \teal{9} & \red{2}
\end{bmatrix},
M_5 =  \begin{bmatrix}
\teal{17} & \red{24} & \blue{1} & \brown{8} & \cyan{15} \\
\red{23} & \blue{5} & \brown{7} & \cyan{14} & \teal{16} \\
\blue{4} & \brown{6} & \cyan{13} & \teal{20} & \red{22} \\
\brown{10} & \cyan{12} & \teal{19} & \red{21} & \blue{3} \\
\cyan{11} & \teal{18} & \red{25} & \blue{2} & \brown{9} 
\end{bmatrix}
\label{odd_magic}
\end{align}
Two different MATLAB codes (one using for-loops and the other one completely vectorized) for obtaining odd magic square are presented here~\cite{clevemagic2}. 

We present the vectorized code, since it is easy to get a matrix interpretation from that code. 
The vectorized MATLAB code is given by \cite{clevemagic2},
\begin{tcolorbox}
	\begin{verbatim}
   [I,J] = ndgrid(1:n);
   A = mod(I+J+(n-3)/2,n);
   B = mod(I+2*J-2,n);
   M = n*A + B + 1;
\end{verbatim}
\end{tcolorbox}
From the code, it immediately follows that the matrices $A$ and $B$ have the entries
$A(i,j) = \bkt{i+j+ \frac{n-3}{2}} \mod n$, and $B(i,j) = \bkt{i+2j-2} \mod n$.

Note $A(i,j) = A(i+1,j-1)$, (With modulo $n$ and index zero being index $n$). Hence, $A$ is a reverse circulant matrix.

Further we have $B(i,j) = B\bkt{i+1,j-\frac{(n+1)}{2}}$ (With modulo $n$ and index zero being index $n$), so $B$ is a reverse $\bkt{\frac{n+1}{2}}$ circulant matrix~\cite{andrade2020spectra}.

Thus we have the decomposition 

\begin{equation}
    M_n = nJ_nX_n + J_n Y_n,
    \label{relation}
\end{equation}
with $X_n$ being a circulant matrix and $Y_n = B + \mathbb{I}\mathbb{I}^T$ being a $\frac{n+1}{2}$ row circulant matrix.\\
 Since $X_n = J_nA$, the entries of the matrix $X_n$ are given by
 \begin{align}
     X_n(i,j) & = A(n-i+1,j) = \mod\bkt{n-i+1+j + \frac{n-3}{2},n}\\
     & = \bkt{-i+j + \frac{n-1}{2} } \mod n.
 \end{align}
 With $m = \frac{n-1}{2}$, the first row of $X_n$ is given by
\begin{equation}
	\mathbf{a} = \begin{bmatrix}m & m+1 & m+2 & \cdots &2m & 0 & 1 & 2 & \cdots & m-1\end{bmatrix}.
\end{equation}

The first column of $X_n$ is given by 

\begin{equation}
	\mathbf{c} = \begin{bmatrix}m & m-1 & m-2 & \cdots &1 & 0 & 2m & 2m-1 & \cdots & m+1\end{bmatrix}^T.
\end{equation}

We consider the matrix $B+\mathbb{I}\mathbb{I}^T = J_nY_n$, such that $Y_n$ is a $\bkt{m+1}$-row circulant~\cite{andrade2020spectra} matrix. Hence, we have
\begin{equation}
Y_n(i,j) = B(n-i+1,j)+1 =1 + (n-i+2j -1 \mod n) = 1 + (-i+2j- 1 \mod n).
\end{equation}

The first row of $Y_n$ is being
\begin{equation}
	\mathbf{b} = \begin{bmatrix}
1 & 3 & \cdots & 2m-1 & 2m+1 & 2 & 4 & \cdots & 2m
\end{bmatrix}.
\end{equation}

The first column of $Y_n$ is given by 

\begin{equation}
	\mathbf{d} = \begin{bmatrix}
1 & 2m+1 & 2m & \cdots & m  & m-1 \cdots & 2
\end{bmatrix}^T.
\end{equation}

Note that 
\begin{equation}
C^{m+1}\mathbf{c} = \mathbf{d} - \mathbb{I}_{n \times 1}. \label{colsxy}
\end{equation}

Since $Y_n$ is $m+1$ circulant matrix, from~\ref{gcirc}, we have $Y_n = R P$ with circulant matrix $R$ having first column same as that of $Y_n$. Hence, from Equation~\eqref{colsxy}, we have \begin{equation}
Y_n = \bkt{C^{m+1}X_n + \mathbb{I}_{n\times n}}P \label{yncxn}
\end{equation}
For the $5\times5$ case, we have the decomposition matrices 
\begin{align}
	\begin{split}
    M_5 & = \begin{bmatrix}
\teal{17} & \red{24} & \blue{1} & \brown{8} & \cyan{15} \\
\red{23} & \blue{5} & \brown{7} & \cyan{14} & \teal{16} \\
\blue{4} & \brown{6} & \cyan{13} & \teal{20} & \red{22} \\
\brown{10} & \cyan{12} & \teal{19} & \red{21} & \blue{3} \\
\cyan{11} & \teal{18} & \red{25} & \blue{2} & \brown{9} 
\end{bmatrix} =
\begin{bmatrix}
\teal{15} & \red{20} & \blue{0} & \brown{5} & \cyan{10}\\
\red{20} & \blue{0} & \brown{5} & \cyan{10} & \teal{15}\\
\blue{0} & \brown{5} & \cyan{10} & \teal{15} & \red{20}\\
\brown{5} & \cyan{10} & \teal{15} & \red{20} & \blue{0}\\
\cyan{10} & \teal{15} & \red{20} & \blue{0} & \brown{5}\\
\end{bmatrix}
+
\begin{bmatrix}
\teal{2} & \red{4} & \blue{1} & \brown{3} & \cyan{5}\\
\brown{3} & \cyan{5} & \teal{2} & \red{4} & \blue{1}\\
\red{4} & \blue{1} & \brown{3} & \cyan{5} & \teal{2}\\
\cyan{5} & \teal{2} & \red{4} & \blue{1} & \brown{3}\\
\blue{1} & \brown{3} & \cyan{5} & \teal{2} & \red{4}\\
\end{bmatrix}
\\
& = \begin{bmatrix}
0 & 0 & 0 & 0 & \bf{1}\\
0 & 0 & 0 & \bf{1} & 0\\
0 & 0 & \bf{1} & 0 & 0\\
0 & \bf{1} & 0 & 0 & 0\\
\bf{1} & 0 & 0 & 0 & 0\\
\end{bmatrix}
\bkt{
5 \overbrace{\begin{bmatrix}
\cyan{2} & \teal{3} & \red{4} & \blue{0} & \brown{1}\\
\brown{1} & \cyan{2} & \teal{3} & \red{4} & \blue{0}\\
\blue{0} & \brown{1} & \cyan{2} & \teal{3} & \red{4}\\
\red{4} & \blue{0} & \brown{1} & \cyan{2} & \teal{3}\\
\teal{3} & \red{4} & \blue{0} & \brown{1} & \cyan{2}\\
\end{bmatrix}}^{X_5}
+
\overbrace{\begin{bmatrix}
\blue{1} & \brown{3} & \cyan{5} & \teal{2} & \red{4}\\
\cyan{5} & \teal{2} & \red{4} & \blue{1} & \brown{3}\\
\red{4} & \blue{1} & \brown{3} & \cyan{5} & \teal{2}\\
\brown{3} & \cyan{5} & \teal{2} & \red{4} & \blue{1}\\
\teal{2} & \red{4} & \blue{1} & \brown{3} & \cyan{5}\\
\end{bmatrix}}^{Y_5}
}.
\end{split}
\end{align}

We rewrite Equation~\eqref{relation} as
\begin{align}
    M_n & = J_n \bkt{\overbrace{nX_n + (m+1)\Ib_n\Ib_n^T}^{S_n} + \underbrace{Y_n - (m+1)\Ib_n\Ib_n^T}_{T_n}},
    \label{UVeqn}
\end{align}
where $\Ib_n$ is the column vector of ones. Note that $S_n$ is a circulant matrix and $T_n$ is $(m+1)$-row circulant matrix.

For the $5 \times 5$ case, we have this decomposition
\begin{align}
    M_5 & = \begin{bmatrix}
0 & 0 & 0 & 0 & \bf{1}\\
0 & 0 & 0 & \bf{1} & 0\\
0 & 0 & \bf{1} & 0 & 0\\
0 & \bf{1} & 0 & 0 & 0\\
\bf{1} & 0 & 0 & 0 & 0\\
\end{bmatrix}
\bkt{
\overbrace{\begin{bmatrix}
\cyan{13} & \teal{18} & \red{23} & \blue{3} & \brown{8}\\
\brown{8} & \cyan{13} & \teal{18} & \red{23} & \blue{3}\\
\blue{3} & \brown{8} & \cyan{13} & \teal{18} & \red{23}\\
\red{23} & \blue{3} & \brown{8} & \cyan{13} & \teal{18}\\
\teal{18} & \red{23} & \blue{3} & \brown{8} & \cyan{13}\\
\end{bmatrix}}^{S_5}
+
\overbrace{\begin{bmatrix}
\blue{-2} & \brown{0} & \cyan{2} & \teal{-1} & \red{1}\\
\cyan{2} & \teal{-1} & \red{1} & \blue{-2} & \brown{0}\\
\red{1} & \blue{-2} & \brown{0} & \cyan{2} & \teal{-1}\\
\brown{0} & \cyan{2} & \teal{-1} & \red{1} & \blue{-2}\\
\teal{-1} & \red{1} & \blue{-2} & \brown{0} & \cyan{2}\\
\end{bmatrix}}^{Y_5}
}.
\end{align}
The main theorem proved in this section is below (Theorem~\ref{main_odd_theorem}).
\begin{tcolorbox}[title=Non-trivial eigenvalues of odd magic squares]
\begin{theorem}
	\label{main_odd_theorem}
	Let $n=2m+1$ and let $$\Lambda = \left\{  \lambda_j = \dfrac{n^2}{2\sin\bkt{j\pi/n}}: j \in\{\pm1,\pm2,\ldots,\pm m\} \right\}$$
	Then the set of non-trivial eigenvalues of an odd magic square generated using the \textbf{Siamese method} is $\frac1n$-approximated by $\Lambda$. 
\end{theorem}
\end{tcolorbox}
A similar result can be found in \cite{bnmn}. The remainder of this section discusses the proof of the above theorem (Theorem~\ref{main_odd_theorem}).
\begin{proof}
From Equation~\eqref{UVeqn}, we have $M_n=J_nS_n+J_nT_n$. \noindent\textbf{We now look at $J_nT_n$ as a perturbation to $J_nS_n$ to get $M_n$ and claim that the eigenvalues of $M_n = J_nS_n+J_nT_n$ are well approximated by the eigenvalues of $J_nS_n$. The eigenvalues and eigenvectors of $J_nS_n$ can be easily computed since $S_n$ is a circulant matrix. We now rely on  the Bauer-Fike theorem (Corollary 2.2 of \cite{eisenstat1998three}) to make the above statement precise.}
\begin{tcolorbox}
\begin{theorem}
\label{Bauer-Fike}
\textbf{Bauer-Fike theorem}: Let $U_n \in \Cb^{n \times n}$ be a diagonalizable matrix. Let $\Lambda_n \in \Cb^{n \times n}$ be a diagonal matrix with eigenvalues of $U_n$ and $W_n \in \Cb^{n \times n}$ be the matrix containing the eigenvectors of $U_n$ such that $U_n = W_n \Lambda_n W_n^{-1}$. Let $\kappa_p\bkt{W_n} = \magn{W_n}_p\magn{W_n^{-1}}_p$ be the condition number of the matrix $W_n$ in $p$-norm. 

If $\mu$ is an eigenvalue of $U_n+V_n$, where $V_n \in \Cb^{n \times n}$, then there exists $\lambda$, an eigenvalue of $U_n$, such that
\begin{align}
	\dfrac{\abs{\mu-\lambda}}{\abs{\lambda}} \leq \kappa_p\bkt{W_n} \magn{U_n^{-1}V_n}_p
\end{align}
\end{theorem}
\end{tcolorbox}
We utilize Theorem~\ref{Bauer-Fike}, by choosing $U_n=J_nS_n$ and $V_n=J_nT_n$. The matrix $J_nS_n$ is a reverse circulant matrix and is diagonalizable. The eigenvalues and eigenvectors of $J_nS_n$ are derived in Section ~\ref{rc_eig}. Further, the eigenvectors can be chosen to be orthonormal. Hence, the matrix $W_n$ in Bauer-Fike theorem (Theorem~\ref{Bauer-Fike}) is orthogonal matrix. Hence from Theorem~\ref{Bauer-Fike}, and choosing $p=2$, we have that if $\mu$ is an eigenvalue of $M_n$, then there exists $\lambda$, an eigenvalue of $J_nS_n$, such that
\begin{align}
	\dfrac{\abs{\mu-\lambda}}{\abs{\lambda}} \leq \magn{\bkt{J_nS_n}^{-1}\bkt{J_nT_n}}_2 = \magn{S_n^{-1}T_n}_2.
\end{align}
\begin{lemma}
	\label{lemma1}
	The eigenvalues of $U_n=J_nS_n$ are given by
	\begin{align}
		\lambda_j = \begin{cases}
		\dfrac{n\bkt{n^2+1}}2 & j=0\\
		\pm\dfrac{n^2}{2\sin\bkt{\pi j/n}} & j \in \{1,2,\ldots,m\}
	\end{cases}
	\end{align}
	The eigenvectors of $U_n=J_nS_n$ are orthonormal.
\end{lemma}
\begin{proof}
	From Lemma~\ref{lem1}, we get the eigenvalues of $X_n$. The eigenvalues of $S_
n$ are the eigenvalues of $X_n$, except the first eigenvalue, $\lambda_0\bkt{S_n}$, which is given by $\lambda_0\bkt{X_n} + \bkt{m+1}$. Note that the matrix $U_n$ is a reverse circulant matrix. Hence, from Section~\ref{rc_eig} eigenpairs of $U_n$ can be obtained in terms of eigenpairs of $S_n$.
\end{proof}
\begin{lemma}
	\label{lemma2}
	\begin{align}
		\magn{S_n^{-1}T_n}_2 = \dfrac1n.
	\end{align}
\end{lemma}

\begin{proof}
First note that $X_n\mathbb{I} = \dfrac{n(n-1)}2\mathbb{I}$ and $Y_n\mathbb{I} = \dfrac{n(n+1)}2\mathbb{I}$. We have $S_n = nX_n+\bkt{m+1}\mathbb{I}\mathbb{I}^T$ and $T_n=Y_n-(m+1)\mathbb{I}\mathbb{I}^T$. By using the Sherman–Morrison formula, we have
\begin{equation}
S_n^{-1} = \frac{1}{n}X_n^{-1}- \frac{n+1}{n^3(n^2+1)} \mathbb{I}\mathbb{I}^T
\end{equation}
Further, $T_n = Y_n - \bkt{\frac{n+1}{2}} \mathbb{I}\mathbb{I}^T$.
Hence, we obtain
\begin{align}
    S_n^{-1}T_n = \frac{1}{n}X_n^{-1}Y_n - \frac{n+1}{n^2(n-1)} \mathbb{I} \mathbb{I}^T.
\end{align}
Let
\begin{equation}
	S_n^{-1}T_n = \dfrac1{n^2}\bkt{nG-\mathbb{I}\mathbb{I}^T},
\end{equation}
where $G = X_n^{-1}Y_n -\frac{2}{n(n-1)} \mathbb{I} \mathbb{I}^T$.

From Equation~\eqref{yncxn}, we have
\begin{equation}
    X_n^{-1}Y_n = \bkt{C^{m+1} + \frac{2}{n(n-1)} \mathbb{I}\mathbb{I}^T}P = C^{m+1}P + \dfrac2{n\bkt{n-1}} \mathbb{I}\mathbb{I}^T.
\end{equation}
This gives us
\begin{equation}
    G = X_n^{-1}Y_n -\frac{2}{n(n-1)} \mathbb{I} \mathbb{I}^T = C^{m+1}P.
\end{equation}
Hence, $G$ is a permutation matrix.


Now we seek bound on $\magn{S_n^{-1}T_n}_2$. We have
\begin{equation}
	\magn{S_n^{-1}T_n}_2^2 = \max_{x \in \Rb^n \backslash \{0\}} \dfrac{\magn{S_n^{-1}T_nx}_2^2}{\magn{x}_2^2}.
\end{equation}
We have
\begin{equation}
	S_n^{-1}T_nx = \dfrac1{n^2}\bkt{nGx - \mathbb{I}\mathbb{I}^Tx} = \dfrac1{n^2}\bkt{nGx - \bkt{x_1+x_2+\cdots+x_n}\mathbb{I}}.
\end{equation}
Hence, we have
\begin{align}
	\magn{S_n^{-1}T_nx}_2^2 & = \dfrac1{n^4}\bkt{\dsum_{j=1}^n \bkt{nx_j-\bkt{x_1+x_2+\cdots+x_n}}^2},\\
	& = \dfrac{\dsum_{j=1}^n\bkt{\bkt{nx_j}^2 + \bkt{x_1+x_2+\cdots+x_n}^2-2nx_j\bkt{x_1+x_2+\cdots+x_n}}}{n^4},\\
	& = \dfrac{n^2 \dsum_{j=1}^n x_j^2 - n\bkt{x_1+x_2+\cdots+x_n}^2}{n^4},\\
	& = \dfrac{\magn{x}_2^2}{n^2} - \dfrac{\bkt{x_1+x_2+\cdots+x_n}^2}{n^3}.
\end{align}
Hence, we have
$$\dfrac{\magn{S_n^{-1}T_nx}_2^2}{\magn{x}_2^2} = \dfrac1{n^2} - \dfrac{\bkt{\mathbb{I}^Tx}^2}{n^3\magn{x}_2^2}.$$
Hence, we obtain that $\magn{S_n^{-1}T_n}_2 = \dfrac1{n}$.
\end{proof}
\begin{tcolorbox}
Hence, from Theorem~\ref{Bauer-Fike}, Lemma~\ref{lemma1} and Lemma~\ref{lemma2}, we have that if $\mu$ is an non-trivial eigenvalue of $M_n$, then there exists
\begin{align}
	\lambda \in \left\{\pm\dfrac{n^2}{2\sin\bkt{\pi/n}},\pm\dfrac{n^2}{2\sin\bkt{2\pi/n}},\ldots,\pm\dfrac{n^2}{2\sin\bkt{m\pi/n}}\right\}
\end{align}
such that
\begin{align}
	\dfrac{\abs{\mu-\lambda}}{\abs{\lambda}} \leq \dfrac1n.
\end{align}
\end{tcolorbox}
This proves Theorem~\ref{main_odd_theorem}.
\end{proof}
\noindent \begin{remark} \label{remark1} We, in fact, numerically observe (see Tables~\ref{magic_3},~\ref{magic_5},~\ref{magic_7},~\ref{magic_9} and~\ref{magic_11}) a stronger version of the above result (Theorem~\ref{main_odd_theorem}), namely, for each eigenvalue of $J_nS_n$, say $\lambda$, there is an eigenvalue of $M_n$, say $\mu$, such that
	\begin{align}
		\dfrac{\abs{\mu-\lambda}}{\abs{\lambda}} \leq \dfrac1n,
	\end{align}
	or equivalently the set of non-trivial eigenvalues of an odd magic square ($M_n$ with $n=2m+1$) generated using the \textbf{Siamese method} is $\frac1n$-well-approximated by the set
	\begin{align}
		\left\{\pm\dfrac{n^2}{2\sin\bkt{\pi/n}},\pm\dfrac{n^2}{2\sin\bkt{2\pi/n}},\ldots,\pm\dfrac{n^2}{2\sin\bkt{m\pi/n}}\right\}.
	\end{align}
\end{remark}
\noindent \begin{remark} \label{remark2} A slightly stronger form of Theorem~\ref{main_odd_theorem} can be proven by diagonalizing the matrix $J_nS_n$ and making use of Gershgorin circle theorem. However, the result still doesn't prove the observation in Remark~\ref{remark1}.
\end{remark}
\begin{table}
	\caption{Eigenvalue of $3 \times 3$ magic square}
	\begin{center}
	\begin{tabular}{|c|c|c|}
		\hline
		Eigenvalues of $M_n$ ($\mu$) & Eigenvalues of $J_nS_n$ ($\lambda$) & Relative error $\bkt{\dfrac{\abs{\mu-\lambda}}{\abs{\lambda}} \leq \dfrac13}$\\
		\hline
        -4.89897949 &      -5.196152422707 & 5.719096e-02\\ 
 \hline 
          4.89897949 &       5.196152422707 & 5.719096e-02\\ 
 \hline 
         15 &      15 & 0\\ 
 \hline 
	\end{tabular}
	\end{center}
	\label{magic_3}
\end{table}
\begin{table}
	\caption{Eigenvalue of $5 \times 5$ magic square}
	\begin{center}
	\begin{tabular}{|c|c|c|}
		\hline
		Eigenvalues of $M_n$ ($\mu$) & Eigenvalues of $J_nS_n$ ($\lambda$) & Relative error $\bkt{\dfrac{\abs{\mu-\lambda}}{\abs{\lambda}} \leq \dfrac15}$\\
		\hline
        -21.27676547 &     -21.266270208801 & 4.935168e-04\\ 
 \hline 
        -13.12628093 &     -13.143277802978 & 1.293199e-03\\ 
 \hline 
         13.12628093 &      13.143277802978 & 1.293199e-03\\ 
 \hline 
         21.27676547 &      21.266270208801 & 4.935168e-04\\ 
 \hline 
         65 &      65 & 0\\ 
 \hline 
	\end{tabular}
	\end{center}
	\label{magic_5}
\end{table}
\begin{table}
	\caption{Eigenvalue of $7 \times 7$ magic square}
	\begin{center}
	\begin{tabular}{|c|c|c|}
		\hline
		Eigenvalues of $M_n$ ($\mu$) & Eigenvalues of $J_nS_n$ ($\lambda$) & Relative error $\bkt{\dfrac{\abs{\mu-\lambda}}{\abs{\lambda}} \leq \dfrac17}$\\
		\hline
       -56.48482668 &     -56.466739338581 & 3.203186e-04\\ 
 \hline 
        -31.08815856 &     -31.336676188403 & 7.930568e-03\\ 
 \hline 
        -25.39666812 &     -25.130063150178 & 1.060901e-02\\ 
 \hline 
         25.39666812 &      25.130063150178 & 1.060901e-02\\ 
 \hline 
         31.08815856 &      31.336676188403 & 7.930568e-03\\ 
 \hline 
         56.48482668 &      56.466739338581 & 3.203186e-04\\ 
 \hline 
        175 &     175 & 0\\
 \hline 
	    \hline
	\end{tabular}
	\end{center}
	\label{magic_7}
\end{table}
\begin{table}
	\caption{Eigenvalue of $9 \times 9$ magic square}
	\begin{center}
	\begin{tabular}{|c|c|c|}
		\hline
		Eigenvalues of $M_n$ ($\mu$) & Eigenvalues of $J_nS_n$ ($\lambda$) & Relative error $\bkt{\dfrac{\abs{\mu-\lambda}}{\abs{\lambda}} \leq \dfrac19}$\\
		\hline
       -118.41407217 &    -118.414078206605 & 5.096328e-08\\ 
 \hline 
        -63.00687636 &     -63.006814987847 & 9.740991e-07\\ 
 \hline 
        -46.47580015 &     -46.765371804360 & 6.192010e-03\\ 
 \hline 
        -41.12470113 &     -41.124777781373 & 1.863973e-06\\ 
 \hline 
         41.12470113 &      41.124777781373 & 1.863973e-06\\ 
 \hline 
         46.47580015 &      46.765371804360 & 6.192010e-03\\ 
 \hline 
         63.00687636 &      63.006814987847 & 9.740991e-07\\ 
 \hline 
        118.41407217 &     118.414078206605 & 5.096329e-08\\ 
 \hline 
        369 &     369 & 0\\
 \hline 
	\end{tabular}
	\end{center}
	\label{magic_9}
\end{table}
\begin{table}
	\caption{Eigenvalue of $11 \times 11$ magic square}
	\small
	\begin{center}
	\begin{tabular}{|c|c|c|}
		\hline
		Eigenvalues of $M_n$ ($\mu$) & Eigenvalues of $J_nS_n$ ($\lambda$) & Relative error $\bkt{\dfrac{\abs{\mu-\lambda}}{\abs{\lambda}} \leq \dfrac1{11}}$\\
		\hline
       -214.74266474 &    -214.742664739496 & 1.005878e-14\\ 
 \hline 
       -111.90424039 &    -111.904240387785 & 6.442010e-12\\ 
 \hline 
        -80.05297263 &     -80.052972642020 & 1.428606e-10\\ 
 \hline 
        -66.51041338 &     -66.510413341849 & 5.782562e-10\\ 
 \hline 
        -61.12213518 &     -61.122135205551 & 4.611025e-10\\ 
 \hline 
         61.12213518 &      61.122135205551 & 4.611032e-10\\ 
 \hline 
         66.51041338 &      66.510413341849 & 5.782560e-10\\ 
 \hline 
         80.05297263 &      80.052972642020 & 1.428577e-10\\ 
 \hline 
        111.90424039 &     111.904240387786 & 6.441375e-12\\ 
 \hline 
        214.74266474 &     214.742664739495 & 9.529374e-15\\ 
 \hline 
        671 &     671 & 0\\ 
 \hline 
	\end{tabular}
	\end{center}
	\label{magic_11}
\end{table}
\begin{table}
	\caption{Eigenvalue of $13 \times 13$ magic square}
	\begin{center}
	\small
	\begin{tabular}{|c|c|c|}
		\hline
		Eigenvalues of $M_n$ ($\mu$) & Eigenvalues of $J_nS_n$ ($\lambda$) & Relative error $\bkt{\dfrac{\abs{\mu-\lambda}}{\abs{\lambda}} \leq \dfrac1{11}}$\\
		\hline
        -353.09013412 &    -353.090134118702 & 1.609884e-16\\ 
  \hline 
        -181.82867798 &    -181.828677981311 & 1.406793e-15\\ 
  \hline 
        -127.42740570 &    -127.427405699228 & 7.449615e-14\\ 
  \hline 
        -102.67515987 &    -102.675159868710 & 7.413024e-13\\ 
  \hline 
         -90.37276161 &     -90.372761608196 & 2.140605e-12\\ 
  \hline 
         -85.12062520 &     -85.120625195150 & 1.493364e-12\\ 
  \hline 
          85.12062520 &      85.120625195150 & 1.493531e-12\\ 
  \hline 
          90.37276161 &      90.372761608196 & 2.139190e-12\\ 
  \hline 
         102.67515987 &     102.675159868710 & 7.421328e-13\\ 
  \hline 
         127.42740570 &     127.427405699227 & 7.159660e-14\\ 
  \hline 
         181.82867798 &     181.828677981311 & 2.032035e-15\\ 
  \hline 
         353.09013412 &     353.090134118702 & 1.931861e-15\\ 
  \hline 
        1105 &    1105 & 0\\ 
  \hline
	\end{tabular}
	\end{center}
	\label{magic_13}
\end{table}
Let 
\begin{align}
	e_n & = \displaystyle\max_{i=1,2,\ldots,n-1} \displaystyle\min_{j=\pm1,\pm2,\ldots,\pm m} \dfrac{\abs{\mu_i-\lambda_j}}{\abs{\lambda_j}}.
\end{align}
We term $e_n$ as maximum relative error in eigenvalues between $M_n$ and $J_nS_n$. Figure~\ref{fig_rel_error} plots $e_n$ versus $n$ for odd values of $n$.
There is a lot to decipher from Figure~\ref{fig_rel_error}. Some observations are enumerated below.
\begin{enumerate}
	\item
	The bound of $1/n$ is very conservative.
	\item
	The maximum relative error for $n \equiv 3\pmod6$ ($n$ a multiple of 3) is much higher than the rest.
	\item
	The maximum relative error is close to machine precision for large prime numbers.
\end{enumerate}
Our current work is not mature enough to offer a deep analysis of the above observations.
\begin{figure}[!htbp]
	\caption{Maximum relative error in eigenvalues between $M_n$ and $J_nS_n$ for odd values of $n$}
	\includegraphics[width=0.9\textwidth]{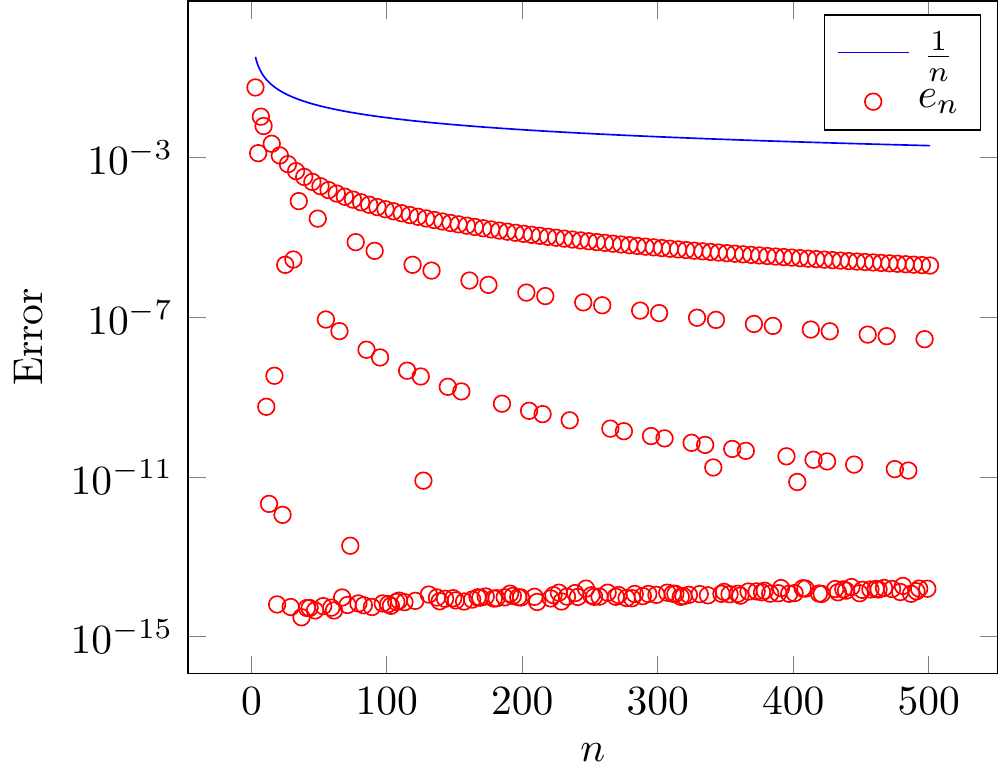}
	\label{fig_rel_error}
\end{figure}
\FloatBarrier
\section{Singly even magic squares}
\label{singly_even}
In this section, we look at the eigenvalues of singly even magic square. The singly even magic squares in MATLAB are constructed using the Strachey method~\cite{ball2016mathematical}. The MATLAB code for obtaining singly even magic square is presented here~\cite{clevemagic2}. We also present it below for the sake of clarity and continuity.
\begin{tcolorbox}
\begin{verbatim}
		function [A] = magic(n)
		% n is twice an odd number n = 2m and m = 2k+1.
		m = n/2; 
		% Smaller magic square
		A = magic (m);   
		% Almost Magic square
		A = [A, A+2*m*m; A+3*m*m, A+m*m];  
		k = (m-1)/2; 
		if (k > 1) 
			I = 1:m; 
			J = [2:k, n-k+2:n]; 
			% Rearranging the column entries
			A([I,I+m],J) = A([I+m,I],J);  
		end 
		I = [1:k, k+2:m];
		A([I,I+m],1) = A([I+m,I],1);
		I = k + 1;
		A([I,I+m],I) = A([I+m,I],I);
\end{verbatim}
\end{tcolorbox}

Let $M_n$ denote the singly even magic square of size $n \times n$, where $n=2m=4k+2$, where $n,m,k \in \Zb^+$. The first two singly even magic squares generated by MATLAB are in Equations~\eqref{M_6} and~\eqref{M_10}.

\begin{align}
	M_6 & =
	\begin{bmatrix}
	    35 &    1 &    6 &   26 &    19 &    24\\
	     3 &   32 &    7 &   21 &    23 &    25\\
	    31 &    9 &    2 &   22 &    27 &    20\\
	     8 &   28 &   33 &   17 &    10 &    15\\
	    30 &    5 &   34 &   12 &    14 &    16\\
	     4 &   36 &   29 &   13 &    18 &    11
	\end{bmatrix}
	\label{M_6}
\end{align}
\begin{align}
	M_{10} & = \begin{bmatrix}
	92 &    99 &     1 &     8   &  15  &   67  &   74  &   51  &   58  &   40 \\
    98 &    80  &    7  &   14  &   16  &   73  &   55  &   57  &   64   &  41 \\
     4 &    81  &   88   &  20  &   22  &   54  &   56  &   63  &   70  &   47 \\
    85 &    87  &   19  &   21  &    3  &   60  &   62  &   69  &   71  &   28 \\
    86 &    93  &   25 &     2  &    9  &   61  &   68  &   75  &   52  &   34 \\
    17 &    24  &   76  &   83  &   90  &   42  &   49  &   26  &   33  &   65 \\
    23 &     5  &  82  &   89  &   91   &  48  &   30  &   32  &   39  &   66 \\
    79 &     6  &  13  &   95  &   97   &  29  &   31  &   38  &   45  &   72 \\
    10 &    12  &  94  &   96   &  78   &  35   &  37   &  44   &  46   &  53 \\
    11 &    18  &  100  &   77  &   84  &   36  &   43  &   50  &   27  &   59 
\end{bmatrix}
\label{M_10}
\end{align}

The singly even magic square thus obtained can be written as follows.
\begin{align}
\begin{split}	
 M_n & = \begin{bmatrix}
	M_m & M_m+2m^2 \mathbb{I}_{m \times m}\\
	M_m+3m^2 \mathbb{I}_{m \times m} & M_m+m^2 \mathbb{I}_{m \times m}
	\end{bmatrix}\\
	& +
	\begin{bmatrix}
		\Ib_{m \times 1}\\
		-\Ib_{m \times 1}
	\end{bmatrix}
	\begin{bmatrix}
		3m^2 \Ib_{1 \times k} & \Ob_{1 \times \bkt{2k+3}} & -m^2 \Ib_{1 \times \bkt{k-1}}
	\end{bmatrix}
	-3m^2uv^T
 \end{split},
\end{align}
where $u,v \in \Rb^{n \times 1}$ with
\begin{align}
	u_j = \begin{cases}
1 & \text{ if } j=k+1\\
-1 & \text{ if } j=n-k\\
0 & \text{ otherwise}
\end{cases}, \hspace{0.5cm} \text{and} \hspace{0.5cm}
v_j = \begin{cases}
1 & \text{ if } j=1\\
-1 & \text{ if } j=k+1\\
0 & \text{ otherwise}
\end{cases}.
\end{align}

\begin{align}
\begin{split}
	M_n & = \begin{bmatrix}M_m+\dfrac{3m^2}2\Ib_{m \times m} & M_m+\dfrac{3m^2}2\Ib_{m \times m}\\M_m+\dfrac{3m^2}2\Ib_{m \times m} & M_m+\dfrac{3m^2}2\Ib_{m \times m}\end{bmatrix} + \begin{bmatrix}-\dfrac{3m^2}2\Ib_{m \times m} & \dfrac{m^2}2\Ib_{m \times m}\\\dfrac{3m^2}2\Ib_{m \times m} & -\dfrac{m^2}2\Ib_{m \times m}\end{bmatrix}\\
	& + \begin{bmatrix}
		\Ib_{m \times 1}\\
		-\Ib_{m \times 1}
	\end{bmatrix}
	\begin{bmatrix}
		3m^2 \Ib_{1 \times k} & \Ob_{1 \times \bkt{2k+3}} & -m^2 \Ib_{1 \times \bkt{k-1}}
	\end{bmatrix}
	- 3m^2uv^T
 \end{split}\\
\begin{split}
& = \begin{bmatrix}
	I_m\\
	I_m
	\end{bmatrix}
	\bkt{M_m+\dfrac{3m^2}2\Ib_{m \times m}}\begin{bmatrix}
	I_m&
	I_m
	\end{bmatrix}
	+\dfrac{m^2}2\begin{bmatrix}
		\Ib_{m \times 1}\\
		-\Ib_{m \times 1}
	\end{bmatrix}
	\begin{bmatrix}
		-3\Ib_{1 \times m} & \Ib_{1 \times m}
	\end{bmatrix}\\
	& + \dfrac{m^2}2\begin{bmatrix}
		\Ib_{m \times 1}\\
		-\Ib_{m \times 1}
	\end{bmatrix}
	\begin{bmatrix}
		6 \Ib_{1 \times k} & \Ob_{1 \times \bkt{2k+3}} & -2 \Ib_{1 \times \bkt{k-1}}
	\end{bmatrix} - 3m^2uv^T.
 \end{split}
\end{align}

\begin{align}
\begin{split}
	M_n & = \begin{bmatrix}
	I_m\\
	I_m
	\end{bmatrix}
	\bkt{M_m+\dfrac{3m^2}2\Ib_{m \times m}}\begin{bmatrix}
	I_m&
	I_m
	\end{bmatrix} - 3m^2uv^T\\
	& +\dfrac{m^2}2\begin{bmatrix}
		\Ib_{m \times 1}\\
		-\Ib_{m \times 1}
	\end{bmatrix}
	\begin{bmatrix}
		3 \Ib_{1 \times k} & -3 \Ib_{1 \times \bkt{k+1}} & \Ib_{1 \times \bkt{k+2}} & - \Ib_{1 \times \bkt{k-1}}
   \end{bmatrix}
\end{split}.
\end{align}
Hence, we see that rank of $M_n$ is at most $m+2$ or equivalently at least $m-2$ of the eigenvalues are zeros. We will now prove the following result on the eigenvalues of $M_n$.

\begin{tcolorbox}[title=Non-trivial eigenvalues of singly even magic squares]
\begin{theorem}
	\label{main_singly_even_theorem}
	Let $M_n$ be a singly even magic square, where $n=2m=4k+2$, generated using the \textbf{Strachey method} and let $$\Lambda = \left\{\lambda_j = \dfrac{n^2}{4\sin\bkt{j\pi/m}}: j \in\{\pm1,\pm2,\ldots,\pm k\} \right\}$$
	\begin{enumerate}
		\item
		Two of the non-trivial eigenpairs of $M_n$ are $\bkt{-3m^2,\begin{bmatrix}\Ib_m\\-\Ib_m\end{bmatrix}}$ and $\bkt{3m^2,\begin{bmatrix}\Ib_m - 3e_{k+1}\\-\bkt{\Ib_m - 3e_{k+1}}\end{bmatrix}}$.
		\item
		$2k-1$ of the non-trivial eigenvalues are zero.
		\item
		The set of remaining $2k$ eigenvalues are $\frac2n$-approximated by $\Lambda$
	\end{enumerate}
\end{theorem}
\end{tcolorbox}


\begin{proof}
Consider the orthogonal matrix $Q_n = \dfrac1{\sqrt{2}}\begin{bmatrix}I_m & -I_m\\I_m & I_m\end{bmatrix}$. The eigenvalues of $M_n$ and the similarity transform $Q_n M_nQ_n^T$ are the same. Let us now evaluate $Q_nM_nQ_n^T$. Note that
\begin{align}
	Q_n\begin{bmatrix}I_m\\I_m\end{bmatrix} = \dfrac1{\sqrt2}\begin{bmatrix}\Ob_{m \times m}\\ 2I_m \end{bmatrix},
\end{align}
\begin{align}
	Q_n\begin{bmatrix}\Ib_{m\times 1}\\-\Ib_{m\times 1}\end{bmatrix} = \dfrac1{\sqrt2}\begin{bmatrix}2\Ib_{m \times 1} \\ \Ob_{m \times 1}\end{bmatrix}.
\end{align}
\begin{align}
	Q_nu = \sqrt2 e_{k+1}.
\end{align}
\begin{align}
	Q_nv = \dfrac1{\sqrt2} \tilde{v},
\end{align}
where $\tilde{v} \in \Rb^{n \times 1}$ with $\tilde{v}_j = \begin{cases}
1 &  \text{ if } j=1,m+1\\
-1 & \text{ if } j=k+1,m+k+1\\
0 & \text{ otherwise}.
\end{cases}$
Thus
\begin{align}
	Q_n\begin{bmatrix}
		\Ib_{m \times 1}\\
		-\Ib_{m \times 1}
	\end{bmatrix} & = \sqrt2 \begin{bmatrix}
	\Ib_{m\times 1}\\
	\Ob_{m\times 1}
	\end{bmatrix},
\end{align}
\begin{align}
	Q_n\begin{bmatrix}
		3 \Ib_{k \times 1} \\ -3 \Ib_{\bkt{k+1} \times 1} \\ \Ib_{\bkt{k+2} \times 1} \\ - \Ib_{\bkt{k-1} \times 1}
	\end{bmatrix} & =\sqrt2
	\begin{bmatrix}
		\Ib_k\\
		-2\\
		-2\\
		-\Ib_{k-1}\\
		2\Ib_k\\
		-1\\
		-1\\
		-2\Ib_{k-1}
	\end{bmatrix}.
\end{align}

Hence,

	\begin{align}
		\boxed{Q_n M_n Q_n^T = \begin{bmatrix}
	m^2A_{m} & m^2B_{m}\\
	\Ob_{m \times m} & 2M_m + 3m^2\Ib_{m \times m}
	\end{bmatrix}
	}.
	\end{align}
where 
\begin{align}
	A_m = \begin{bmatrix}\Ib_{m \times k} & -2\Ib_{m \times 1} & -2\Ib_{m \times 1} & -\Ib_{m \times \bkt{k-1}}\end{bmatrix}-3e_{k+1}w^T,
\end{align}
and
\begin{align}
	B_m=\begin{bmatrix}2\Ib_{m \times k} & - \Ib_{m \times 1} & - \Ib_{m \times 1} & -2\Ib_{m \times \bkt{k-1}}\end{bmatrix}-3e_{k+1}w^T.
\end{align}
where $w,e_{k+1} \in \Rb^{m \times 1}$ with
\begin{align}
	w_j = \begin{cases}
1 & \text{ if } j=1\\
-1 & \text{ if } j=k+1\\
0 & \text{ otherwise}.
\end{cases}
\end{align} and $e_{k+1}$ is a unit vector whose $\bkt{k+1}^{th}$ entry is $1$. Hence, the eigenvalues of $M_n$ are the eigenvalues of $Q_nM_nQ_n^T$, which inturn are the eigenvalues of $m^2A_m$ and $2M_m + 3m^2\Ib_{m \times m}$.

Note that $A_m$ is a rank $2$ matrix, since
\begin{align}
	A_m = \Ib_{m \times 1}\begin{bmatrix}\Ib_{1 \times k} & -2 & -2 & -\Ib_{1 \times \bkt{k-1}}\end{bmatrix} -3e_{k+1}w^T.
\end{align}

Further, note that
\begin{align}
	A_m \Ib_{m \times 1} & = \Ib_{m \times 1}\begin{bmatrix}\Ib_{1 \times k} & -2 & -2 & -\Ib_{1 \times \bkt{k-1}}\end{bmatrix} \Ib_{m \times 1} -3e_{k+1}w^T \Ib_{m \times 1},\\
	& = -3\Ib_{m \times 1}.
\end{align}
since $w^T \Ib = 0$. Also, note that the trace of $A_m$ is zero. Hence, the non-zero eigenvalues of $A_m$ are $\pm3$. Hence, two of the non-zero eigenvalues of $M_n$ are $\pm3m^2$, i.e., $\pm \dfrac{3n^2}4$. The remaining non-zero eigenvalues of $M_n$ are the eigenvalues of $2M_m+3m^2\Ib_{m \times m}$. From the previous section, we have that the eigenvalues of $M_m$ to be its magic sum, $\dfrac{m\bkt{m^2+1}}2$ and the remaining are $\frac1m$-approximated by
\begin{align}
\left\{\pm\dfrac{m^2}{\sin\bkt{j\pi/m}} \text{ where }j\in\{1,2,\ldots,k\} \right\}.
\end{align}
Of course the magic sum of $M_n = \dfrac{n\bkt{n^2+1}}2$ is nothing but $2\bkt{\dfrac{m\bkt{m^2+1}}2} + 3m^3$. The remaining non-trivial non-zero eigenvalues of $M_n$ are $\frac2n$-approximated by
\begin{align}
	\left\{\pm \dfrac{n^2}{4\sin\bkt{\pi/m}},\pm \dfrac{n^2}{4\sin\bkt{2\pi/m}},\ldots,\pm \dfrac{n^2}{4\sin\bkt{k\pi/m}} \right\}.
\end{align}
\end{proof}

\noindent \begin{remark} \label{remark3} If we could prove the stronger $\frac1n$-well-approximated remark as in Remark~\ref{remark1}, then Theorem~\ref{main_singly_even_theorem} could also be further strengthened appropriately.
\end{remark}
\section{Doubly even magic squares}
\label{doubly_even}
The doubly even magic squares in MATLAB are generated by a criss-cross method as presented here~\cite{clevemagic2}. We present the code below for the sake of clarity and continuity.
\begin{tcolorbox}
\begin{verbatim}
	    M = reshape(1:n^2,n,n)';
	    [I,J] = ndgrid(1:n);
	    K = fix(mod(I,4)/2) == fix(mod(J,4)/2);
	    M(K) = n^2+1 - M(K);
\end{verbatim}
\end{tcolorbox}
The first two doubly even magic squares generated by MATLAB are in Equations~\eqref{M_4} and~\eqref{M_8}.
\begin{align}
	\label{M_4}
	M_4 & = \begin{bmatrix}
	16 & 2 & 3 & 13\\
     5 & 11 & 10 & 8\\
     9 & 7 & 6 & 12\\
     4 & 14 & 15 & 1
	 \end{bmatrix}\\
	 	\label{M_8}
	 M_8 & =\begin{bmatrix}
     64  &   2   &   3  &    61 &    60 &    6  &    7  &    57\\
      9  &   55  &   54 &    12 &    13 &    51 &    50 &   16\\
     17  &   47  &   46 &    20 &    21 &    43 &    42 &    24\\
     40  &   26  &   27 &    37 &    36 &    30 &    31 &    33\\
     32  &   34  &   35 &    29 &    28 &    38 &    39 &    25\\
     41  &   23  &   22 &    44 &    45 &    19 &    18 &    48\\
     49  &   15  &   14 &    52 &    53 &    11 &    10 &    56\\
      8  &   58  &   59 &     5 &     4 &    62 &    63 &     1
	 \end{bmatrix}
\end{align}
 The main theorem proved in this section is below (Theorem~\ref{main_doubly_even_theorem}).
\begin{tcolorbox}[title=Non-trivial eigenvalues of doubly even magic squares]
\begin{theorem}
	\label{main_doubly_even_theorem}
	The rank of the doubly even magic square (where $n = 4m$) generated by MATLAB's \textbf{criss-cross} method is $3$ and the non-zero non-trivial eigenvalues are
	$$\pm \dfrac{n}2 \sqrt{\dfrac{n^3-n}3}$$
	The corresponding eigenvectors are $\sqrt{n^2-1}v \pm \sqrt{3n}u$, where
\setcounter{MaxMatrixCols}{12}
	$$u = \begin{bmatrix} w\\
	-\text{flip}\bkt{w} \end{bmatrix}_{4m \times 1}$$ and
	$$v^T = \begin{bmatrix} 1 & -1 & -1 & 1 & \cdots & 1 & -1 & -1 & 1 \end{bmatrix}_{1 \times 4m}$$
	with $w^T = \begin{bmatrix}(n-1) & -(n-3) & -(n-5) & (n-7) & \cdots & (-1)^m\end{bmatrix}$.
\end{theorem}
\end{tcolorbox}
The proof of the above theorem (Theorem~\ref{main_doubly_even_theorem}) constitutes this section.

Let $P_4 = I_4+J_4$, where $I_4 \in \Rb^{4 \times 4}$ is the identity matrix and $J_4 \in \Rb^{4 \times 4}$ is the reverse identity matrix. If $S_n$ denotes the $n \times n$ matrix with numbers arranged row wise from $1$ to $n^2$, i.e., $S_n\bkt{i,j} = j+(i-1)n$, it is fairly easy to show that
\begin{align}
\begin{split}
	M_n\bkt{i,j} & = S_n\bkt{i,j} \bkt{1-P_4\bkt{i\bmod4,j\bmod4}}\\
	 & + S_n\bkt{n+1-i,n+1-j} P_4\bkt{i\bmod4,j\bmod4}.
\end{split}
\end{align}
It is a fairly elementary algebra to show that 
\begin{equation}
	M_n = \dfrac{n^2+1}2 \Ib \Ib^T + \dfrac{n}2 uv^T + \dfrac{vu^T}2,
\end{equation}
where
\begin{equation}u = \begin{bmatrix} w\\
	-\text{flip}\bkt{w} \end{bmatrix}_{4m \times 1},
\end{equation}
where $w^T = \begin{bmatrix}(n-1) & -(n-3) & -(n-5) & (n-7) & \cdots & (-1)^m\end{bmatrix}$ and
\begin{equation}
	v^T = \begin{bmatrix} 1 & -1 & -1 & 1 & \cdots & 1 & -1 & -1 & 1 \end{bmatrix}_{1 \times 4m}.
\end{equation}
This gives us
\begin{align}
\begin{split}
	M_n & = \begin{bmatrix}
		\Ib & u & v
	\end{bmatrix}
	\begin{bmatrix}
		\dfrac{n^2+1}2 & 0 & 0\\
		0 & 0 & \dfrac{n}2\\
		0 & \dfrac12 & 0
	\end{bmatrix}
	\begin{bmatrix}
		\Ib^T \\ u^T \\ v^T
	\end{bmatrix},\\
	& = USU^T. \label{usut}
\end{split}
\end{align}
Here, $U =  \begin{bmatrix} \Ib & u & v \end{bmatrix}$ and $S = \dfrac12 \begin{bmatrix}n^2+1 & 0 & 0\\ 0 & 0 & n\\ 0 & 1 & 0\end{bmatrix}$.

We have
\begin{align}
	D & = U^TU = \dfrac{n}3 \begin{bmatrix}
3 & 0 & 0\\
0 & n^2-1 & 0\\
0 & 0 & 3
\end{bmatrix}.
\end{align}

Since $M_n$ is rank $3$, we immediately obtain that $n-3$ eigenvalues are zeros and the corresponding eigenvectors are the vectors orthogonal to $\Ib$, $u$ and $v$. Now let's obtain the non-zero eigenvalues of $M_n$. Let $\bkt{\lambda,x}$ be a non-zero non-trivial eigenpair of $M_n$. This implies $M_nx = \lambda x$. Since eigenvector is in the range of $M_n$, and range of $M_n$ is same as range of $U$ (from Equation~\eqref{usut}), we have $x=Uy$, where $y \in \Rb^{3 \times 1}$. Hence, we have
\begin{align}
	USU^TUy & = \lambda U y, \label{ulambda}
 \end{align}
 Multiplying by the left inverse $(U^TU)^{-1}U^T$ on both sides of \eqref{ulambda}, we get
 \begin{align}
	SDy = \lambda y.
\end{align}
Hence, $\bkt{\lambda,y}$ is an eigenpair of $SD \in \Rb^{3 \times 3}$. We have
\begin{align}
	SD & = \dfrac{n}{6}\begin{bmatrix}
	3\bkt{n^2+1} & 0 & 0\\
	0 & 0 & 3n\\
	0 & n^2-1 & 0
	\end{bmatrix}.
\end{align}
The eigenpairs corresponding to non-zero eigenvalues are
\begin{enumerate}
	\item
	$\bkt{\dfrac{n\bkt{n^2+1}}2,\Ib_n}$
	\item
	$\bkt{\dfrac{n}2\sqrt{\dfrac{n^3-n}3},\sqrt{3n}u+\sqrt{n^2-1}v}$
	\item
	$-\bkt{\dfrac{n}2\sqrt{\dfrac{n^3-n}3},-\sqrt{3n}u+\sqrt{n^2-1}v}$.
\end{enumerate}
This gives us the non-zero eigenvalues to be $\dfrac{n\bkt{n^2+1}}2, \pm \dfrac{n}2\sqrt{\dfrac{n^3-n}3}$ and thereby proving Theorem~\ref{main_doubly_even_theorem}.
\section{Conclusion}
\label{conclusion}
Eigenvalues of the magic squares obtained using the magic() function in MATLAB are investigated.  We obtain approximations of eigenvalues of odd and singly even magic squares with error bounds. Further, the eigenvalues of doubly even magic squares are exactly obtained. The approximation of the spectra involves some interesting connections with the spectrum of g-circulant matrices and the use of the Bauer-Fike theorem. We also highlight some observations which we have not been able to prove and hope to address at least some of them in the future.
\section{Appendix}
\label{appendix}
Let $n$ be an odd number of the form $2m+1$, where $m \in \Zb$.
Then we have the following lemma.
\begin{lemma}
\label{lem1}
Let $X_n \in \Rb^{n \times n}$ be a circulant matrix whose first row is
\begin{equation}
    \bf{a} =\begin{bmatrix}
m & m+1 & m+2 & \cdots & 2m & 0 & 1 & \cdots & m-1
    \end{bmatrix}.
\end{equation}
We then have the eigenvalues of $X_n$ to be given by
\begin{equation}
    \lambda_j = \begin{cases}
    \dfrac{n\bkt{n-1}}2 & \text{ if }j=0\\
    \dfrac{\bkt{-1}^{j} in}2\csc\bkt{\dfrac{\pi j}n} & \text{ if }j \in \{1,2,\ldots,n-1\}.
    \end{cases}
\end{equation}
\end{lemma}
\begin{proof}
The $j^{th}$ eigenvalues is given by
\begin{equation}
    \lambda_j  = {\bf{a}} v_j.
\end{equation}
Hence, we have
\begin{align}
\begin{split}
    \lambda_j & = \dsum_{k=0}^{2m} \bkt{\bkt{m+k}\mod n}\omega^{kj} = \omega^{-mj}\bkt{\dsum_{k=0}^{2m} \bkt{\bkt{m+k}\mod n}\omega^{(m+k)j}}\\
    & = \omega^{-mj}\bkt{\dsum_{k=0}^{2m} k\omega^{kj}}.
\end{split}
\end{align}
For $j=0$, we have $\lambda_n = 0+1+\cdots+2m = \dfrac{n\bkt{n-1}}2$. Let $j \in \{1,2,\ldots,n-1\}$. We then have
\begin{align}
\begin{split}
    \lambda_j & =\omega^{-mj}\bkt{\dsum_{k=0}^{2m} k\omega^{kj}} = \dfrac{n\omega^{-mj}}{\omega^{j}-1} = \dfrac{n\exp\bkt{\dfrac{2\pi imj}n}}{\exp\bkt{-\dfrac{2\pi i j}n}-1} = \dfrac{n\exp\bkt{\dfrac{2\pi ij(n-1)}{2n}}}{\exp\bkt{-\dfrac{2\pi i j}n}-1},\\
    &= \dfrac{ \exp \bkt{\dfrac{\pi i j}{n}} n\exp\bkt{\dfrac{2\pi ij(n-1)}{2n}}}{ \exp\bkt{\dfrac{\pi i j}{n}} \bkt{\exp\bkt{-\dfrac{2\pi i j}n}-1}},  \\
   &=\dfrac{ n\exp\bkt{\dfrac{2\pi ij(n)}{2n}}}{\bkt{\exp\bkt{-\dfrac{\pi i j}n}-\exp\bkt{-\dfrac{\pi i j}n}}} = \dfrac{ n\exp\bkt{\pi i j}}{-2i \frac{\exp\bkt{\dfrac{\pi i j}{n}}-\exp\bkt{-\dfrac{\pi i j}{n}}}{2i}}, \\
    & = \dfrac{(-1)^{j+1} n}{2i\sin\bkt{\dfrac{\pi j}n}} = \bkt{-1}^{j} \dfrac{i n}{2\sin\bkt{\dfrac{\pi j}n}} = \boxed{\dfrac{(-1)^j i n}2 \csc\bkt{\dfrac{\pi j}n}}.
\end{split}
\end{align}
Hence,
\begin{equation}
	X_n = \dfrac{F_n \Lambda_n F_n^*}n,
\end{equation}
where $\Lambda_n(j,j) = \lambda_j  = \boxed{\dfrac{(-1)^j i n}2 \csc\bkt{\dfrac{\pi j}n}}$.
\end{proof}

\bibliographystyle{siamplain}
\bibliography{references}

\begin{thebibliography}{10}

\bibitem{adetokunbo20163d}
{\sc P.~Adetokunbo, A.~A. Al-Shuhail, and S.~Al-Dossary}, {\em 3d seismic edge
  detection using magic squares and cubes}, Interpretation, 4 (2016),
  pp.~T271--T280.

\bibitem{anderson2001mathematical}
{\sc D.~Anderson}, {\em Mathematical rootsl: Magic squares: Discovering their
  history and their magic}, Mathematics Teaching in the Middle School, 6
  (2001), pp.~466--470.

\bibitem{andrade2020spectra}
{\sc E.~Andrade, L.~Arrieta, C.~Manzaneda, and M.~Robbiano}, {\em On the
  spectra of some g-circulant matrices and applications to nonnegative inverse
  eigenvalue problem}, Linear Algebra and its Applications, 590 (2020),
  pp.~1--21.

\bibitem{ball2016mathematical}
{\sc W.~R. Ball and H.~S.~M. Coxeter}, {\em Mathematical recreations \&
  essays}, in Mathematical Recreations \& Essays, University of Toronto Press,
  2016.

\bibitem{benefiel2012magic}
{\sc R.~R. Benefiel}, {\em Magic squares, alphabet jumbles, riddles and more:
  The culture of word-games among the graffiti of pompeii}, in The muse at
  play, De Gruyter, 2012, pp.~65--80.

\bibitem{cammann1969islamic}
{\sc S.~Cammann}, {\em Islamic and indian magic squares. part i}, History of
  Religions, 8 (1969), pp.~181--209.

\bibitem{eisenstat1998three}
{\sc S.~C. Eisenstat and I.~C. Ipsen}, {\em Three absolute perturbation bounds
  for matrix eigenvalues imply relative bounds}, SIAM Journal on Matrix
  Analysis and Applications, 20 (1998), pp.~149--158.

\bibitem{hunter2010some}
{\sc J.~J. Hunter}, {\em Some stochastic properties of “semi-magic” and
  “magic” markov chains}, Linear algebra and its applications, 433 (2010),
  pp.~893--907.

\bibitem{lee2006linear}
{\sc M.~Lee, E.~Love, E.~Wascher, et~al.}, {\em Linear algebra of magic
  squares}, Undergraduate Research, Central Michigan University, Mount
  Pleasant, Mich, USA,  (2006).

\bibitem{liu2012study}
{\sc Y.~LIU and J.~bo~Liu}, {\em Study and application of magic square group in
  process of image scrambling}, Computer Technology and Development, 22 (2012),
  pp.~119--122.

\bibitem{loly2009magic}
{\sc P.~Loly, I.~Cameron, W.~Trump, and D.~Schindel}, {\em Magic square
  spectra}, Linear algebra and its applications, 430 (2009), pp.~2659--2680.

\bibitem{mattingly2000even}
{\sc R.~B. Mattingly}, {\em Even order regular magic squares are singular}, The
  American Mathematical Monthly, 107 (2000), pp.~777--782.

\bibitem{clevemagic2}
{\sc C.~Moler}, {\em Magic squares, part 2, algorithms}, 2012,
  \url{https://blogs.mathworks.com/cleve/2012/11/05/magic-squares-part-2-algorithms/}.

\bibitem{moler2004numerical}
{\sc C.~B. Moler}, {\em Numerical computing with MATLAB}, SIAM, 2004.

\bibitem{bnmn}
{\sc B.~Newman}, {\em Matlab magic}, contributed presentation, 1997 MATLAB
  Conference, Sydney, Australia,  (October 23-24, 1997).

\bibitem{ranjani2017data}
{\sc J.~J. Ranjani}, {\em Data hiding using pseudo magic squares for embedding
  high payload in digital images}, Multimedia Tools and Applications, 76
  (2017), pp.~3715--3729.

\bibitem{van1990magic}
{\sc A.~Van Den~Essen}, {\em Magic squares and linear algebra}, The American
  Mathematical Monthly, 97 (1990), pp.~60--62.

\bibitem{weisstein2002magic}
{\sc E.~W. Weisstein}, {\em Magic square}, https://mathworld.wolfram.com/,
  (2002).

\end{thebibliography}
\end{document}